\documentclass[11pt]{article}
\usepackage[T1]{fontenc}
\usepackage{lmodern,amsmath,amsthm,amsfonts,amssymb,graphicx,float,calc,microtype,thmtools,underscore,mathtools,cmbright} 
\usepackage[usenames,dvipsnames,svgnames,table]{xcolor}
\usepackage[unicode=true]{hyperref}
\hypersetup{ 
colorlinks,
linkcolor={black},
citecolor={black},
urlcolor={blue!60!black},
pdftitle={Notes on Graph Product Structure Theory},
pdfauthor={Dvo{\v{r}}{\'a}k -- Huynh -- Joret -- Liu -- Wood}}

\usepackage{enumitem}
\setlist[itemize]{topsep=0ex,itemsep=0ex,parsep=0.4ex}
\setlist[enumerate]{topsep=0ex,itemsep=0ex,parsep=0.4ex}

\usepackage[noabbrev,capitalise]{cleveref}
\crefname{lem}{Lemma}{Lemmas}
\crefname{thm}{Theorem}{Theorems}
\crefname{cor}{Corollary}{Corollaries}
\crefname{prop}{Proposition}{Propositions}
\crefname{conj}{Conjecture}{Conjectures}
\crefname{ques}{Question}{Questions}
\crefname{open}{Open Problem}{Open Problems}

\crefformat{equation}{(#2#1#3)}
\Crefformat{equation}{Equation #2(#1)#3}

\newcommand{\GG}{\mathcal{G}}

\DeclareMathOperator{\col}{col}
\DeclareMathOperator{\dist}{dist}

\usepackage[longnamesfirst,numbers,sort&compress]{natbib}
\makeatletter
\def\NAT@spacechar{~}
\makeatother
\usepackage[bmargin=25mm,tmargin=25mm,lmargin=35mm,rmargin=35mm]{geometry}
\allowdisplaybreaks

\DeclarePairedDelimiter{\ceil}{\lceil}{\rceil}
\renewcommand{\ge}{\geqslant}
\renewcommand{\le}{\leqslant}
\renewcommand{\geq}{\geqslant}
\renewcommand{\leq}{\leqslant}
\DeclareMathOperator{\tw}{tw}
\DeclareMathOperator{\pw}{pw}
\renewcommand{\thefootnote}{\fnsymbol{footnote}}
\allowdisplaybreaks

\newcommand{\doi}[1]{doi:\,\href{https://dx.doi.org/#1}{#1}}
\theoremstyle{plain}
\newtheorem{thm}{Theorem}
\newtheorem{lem}[thm]{Lemma}
\newtheorem{cor}[thm]{Corollary}

\newtheorem{open}[thm]{Open Problem}
\theoremstyle{definition}
\newtheorem{conj}[thm]{Conjecture}

\newcommand{\PP}{\mathcal{P}}

\makeatother

\begin{document}

\author{
\quad Zden{\v{e}}k Dvo{\v{r}}{\'a}k\,\footnotemark[4]
\quad Tony Huynh\,\footnotemark[5]
\quad Gwena\"el Joret\,\footnotemark[3]\\
Chun-Hung Liu\,\footnotemark[2]
\quad 
David~R.~Wood\,\footnotemark[5]
}

\footnotetext[4]{Charles University, Prague, Czech Republic (\texttt{rakdver@iuuk.mff.cuni.cz}).  Supported by project 17-04611S (Ramsey-like aspects of graph coloring) of Czech Science Foundation.}

\footnotetext[5]{School of Mathematics, Monash   University, Melbourne, Australia  (\texttt{\{Tony.Huynh2,
david.wood\}@monash.edu}). Research supported by the Australian Research Council.}

\footnotetext[3]{D\'epartement d'Informatique, Universit\'e libre de Bruxelles, Brussels, Belgium (\texttt{gjoret@ulb.ac.be}). Research supported by the Wallonia-Brussels Federation of Belgium, and by the Australian Research Council.}

\footnotetext[2]{Department of Mathematics, Texas A\&M University, College Station, Texas, USA (\texttt{chliu@math.tamu.edu}). Partially supported by the NSF under Grant No.\ DMS-1929851.}

\sloppy

\title{\textbf{Notes on Graph Product Structure Theory\footnote{This research was initiated at the MATRIX research  program ``Structural Graph Theory Downunder'', November 2019 (\texttt{https://www.matrix-inst.org.au/events/structural-graph-theory-downunder/}).}}}

\maketitle

\begin{abstract}
It was recently proved that every planar graph is a subgraph of the strong product of a path and a graph with bounded treewidth. This paper surveys generalisations of this result for graphs on surfaces, minor-closed classes, various non-minor-closed classes, and graph classes with polynomial growth. We then explore how graph product structure might be applicable to more broadly defined graph classes. In particular, we characterise when a graph class defined by a cartesian or strong product has bounded or polynomial expansion. We then explore graph product structure theorems for various geometrically defined graph classes, and present several open problems. 
\end{abstract}

\section{Introduction}
\label{Intro}

\renewcommand{\thefootnote}{\arabic{footnote}}

Studying the structure of graphs is a fundamental topic of broad interest in combinatorial mathematics. At the forefront of this study is the Graph Minor Theorem of \citet{RS-GraphMinors}, described by \citet{Diestel4} as ``among the deepest theorems mathematics has to offer''. At the heart of the proof of this theorem is the Graph Minor Structure Theorem, which shows that any graph in a minor-closed family\footnote{A graph $H$ is a \emph{minor} of a graph $G$ if a graph isomorphic to $H$ can be obtained from a subgraph of $G$ by contracting edges. A class of graphs $\GG$ is \emph{minor-closed} if for every graph $G\in\GG$ every minor of $G$ is in $\GG$, and some graph is not in $\GG$. A graph $G$ is \emph{$H$-minor-free} if $H$ is not a minor of $G$.} can be constructed using four ingredients: graphs on surfaces, vortices, apex vertices, and clique-sums. Graphs of bounded genus, and in particular planar graphs are basic building blocks in graph minor structure theory. Indeed, the theory says nothing about the structure of planar graphs. So it is natural to ask whether planar graphs can be described in terms of some simpler graph classes. 
In a recent breakthrough, \citet{DJMMUW19} provided an answer to this question by showing  that every planar graph is a subgraph of the strong product\footnote{The \emph{cartesian product} of graphs $A$ and $B$, denoted by $A\square B$, is the graph with vertex set $V(A)\times V(B)$, where distinct vertices $(v,x),(w,y)\in V(A)\times V(B)$ are adjacent if:
$v=w$ and $xy\in E(B)$; or
$x=y$ and $vw\in E(A)$.
The \emph{strong product} of graphs $A$ and $B$, denoted by
$A\boxtimes B$, is the graph with vertex set $V(A)\times V(B)$, where distinct vertices $(v,x),(w,y)\in V(A)\times V(B)$ are adjacent if:
$v=w$ and $xy\in E(B)$; or
$x=y$ and $vw\in E(A)$; or
$vw\in E(A)$ and $xy\in E(B)$.
If $X$ is a subgraph of $A\square B$, then the \emph{projection} of $X$ into $A$ is the set of vertices $v\in V(A)$ such that $(v,w)\in V(X)$ for some $w\in V(B)$.}
 of a graph of bounded treewidth\footnote{A \emph{tree decomposition} of a graph $G$ is a collection $(B_x\subseteq V(G):x\in V(T))$ of subsets of $V(G)$ (called \emph{bags}) indexed by the nodes of a tree $T$, such that (i)  for every edge $uv\in E(G)$, some bag $B_x$ contains both $u$ and $v$, and (ii)  for every vertex $v\in V(G)$, the set $\{x\in V(T):v\in B_x\}$ induces a non-empty (connected) subtree of $T$. The \emph{width} of a tree decomposition is the size of the largest bag minus 1. The \emph{treewidth} of a graph $G$, denoted by $\tw(G)$, is the minimum width of a tree decomposition of $G$. See~\citep{Reed97,HW17,Bodlaender-AC93,Bodlaender-TCS98,Reed03} for surveys on treewidth. A \emph{path decomposition} is a tree decomposition where the underlying tree is a path. The \emph{pathwidth} of a graph $G$, denoted by $\pw(G)$, is the minimum width of a path decomposition of $G$.} and a path. 

\begin{thm}[\citep{DJMMUW19}]
\label{PlanarProduct}
Every planar graph is a subgraph of:
\begin{enumerate}[label=(\alph*)]
\item $H\boxtimes P$ for some graph $H$ of treewidth at most $8$ and for some path $P$;
\item $H\boxtimes P\boxtimes K_3$ for some graph $H$ of treewidth at most $3$ and for some path $P$. 
\end{enumerate}
\end{thm}

This \emph{graph product structure theorem} is attractive since it describes planar graphs in terms of graphs of bounded treewidth, which are considered much simpler than planar graphs. For example, many NP-complete problem remain NP-complete on planar graphs but are polynomial-time solvable on graphs of bounded treewidth. 

Despite being only 10 months old, \cref{PlanarProduct} is already having significant impact. Indeed, it has been used to solve two major open problems and make additional progress on two other longstanding problems:

\begin{itemize}
\item \citet{DJMMUW19} use \cref{PlanarProduct} to show that planar graphs
have queue layouts with a bounded number of queues, solving a 27 year old problem of
\citet{HLR92}. 

\item \citet{DEJMW} use \cref{PlanarProduct} to show that planar graphs can be nonrepetitively coloured with a bounded number of colours, solving a 17 year old problem of \citet{AGHR02}.

\item \citet{DFMS20} use \cref{PlanarProduct} to prove the best known results on  $p$-centred colourings of planar graphs, reducing the bound from $O(p^{19})$ to $O(p^3\log p)$.  

\item \citet{BGP20} use \cref{PlanarProduct} to give
more compact graph encodings of planar graphs.
In graph-theoretic terms, this implies the existence of a graph with $n^{4/3+o(1)}$ vertices that contains each planar graph with at most $n$ vertices as an induced subgraph, This work improves a sequence of results that goes back 27 years to the introduction of implicit labelling schemes by \citet{KNR92}.

\end{itemize}

The first goal of this paper is to introduce several product structure theorems that have been recently established, most of which generalise \cref{PlanarProduct}. First \cref{MinorClosedClasses} considers minor-closed classes. Then 
\cref{NonMinorClosedClasses} considers several examples of non-minor-closed classes. \cref{PolynomialGrowth} introduces the notion of graph classes with polynomial growth and their characterisation in terms of strong products of paths due to \citet{KL07}. We prove an extension of this result for strong products of graphs of given pathwidth.

The remaining sections explore how graph product structure might be applicable to more broadly defined graph classes. The following definition by \citet{Sparsity} provides a setting for this study\footnote{Let $d_G(u,v)$ be the distance between vertices $u$ and $v$ in a graph $G$. For a vertex $v$ in a graph $G$ and $r\in\mathbb{N}$, let $N^r_G(v)$ be the set of vertices of $G$ at distance exactly $r$ from $v$, and let $N^r_G[v]$ be the set of vertices at distance at most $r$ from $v$. The set $N^r_G[v]$ is called an \emph{$r$-ball}. We drop the subscript $G$ when the graph is clear from the context.}. A graph class $\GG$ has \emph{bounded expansion} with \emph{expansion function} $f:\mathbb{Z}^+\to\mathbb{R}$ if, for every graph $G\in\GG$
and for all disjoint subgraphs $B_1,\dots,B_t$ of radius at most $r$ in $G$,
every subgraph of the graph obtained from $G$ by contracting each $B_i$ into a vertex has average degree at most $f(r)$.  When $f(r)$ is a constant, $\GG$ is contained in a proper minor-closed class. As $f(r)$ is allowed to grow with $r$ we obtain larger and larger graph classes. A graph class $\GG$ has \emph{linear expansion} if $\GG$ has bounded expansion with an expansion function in $O(r)$. A graph class $\GG$ has \emph{polynomial expansion} if $\GG$ has bounded expansion with an expansion function in $O(r^c)$, for some constant $c$. 

We characterise when a graph class defined by a cartesian or strong product has bounded or polynomial expansion. For $\star \in\{ \boxtimes, \square\}$ and for hereditary\footnote{A class of graphs is \emph{hereditary} if it is closed under induced subgraphs.} graph classes $\GG_1$ and $\GG_2$, let  
\begin{align*}
\GG_1\star \GG_2 &:= \{ G : G \subseteq G_1 \star G_2, G_1 \in \GG_1 , G_2 \in \GG_2 \}.
\end{align*}
Note that $\GG_1\star \GG_2$ is hereditary. 
\cref{PolynomialExpansion,BoundedExpansion} characterise when $\GG_1 \star \GG_2$ has bounded or polynomial expansion. In related work,  \citet{Wood-ProductMinor} characterised when $\GG_1 \square \GG_2$ has bounded Hadwiger number, and  \citet{Tarik} characterised  when $\GG_1 \boxtimes \GG_2$ has bounded Hadwiger number.

\cref{HigherDimensions} explores graph product structure theorems for various geometrically defined graph classes. We show that multi-dimensional unit-disc graphs have a product structure theorem, and discusses whether two other naturally defined graph classes might have product structure theorems. We finish with a number of open problems in \cref{OpenProblems}.

\section{Minor-Closed Classes}
\label{MinorClosedClasses}

Here we survey results generalising \cref{PlanarProduct} for minor-closed classes. First consider graphs embeddable on a fixed surface\footnote{The \emph{Euler genus} of the orientable surface with $h$ handles is $2h$. The \emph{Euler genus} of  the non-orientable surface with $c$ cross-caps is $c$. The \emph{Euler genus} of a graph $G$ is the minimum Euler genus of a surface in which $G$ embeds (with no crossings). See~\citep{MoharThom} for background on embeddings of graphs on surfaces.}. 

\begin{thm}[\citep{DJMMUW19}]
\label{SurfaceProduct}
Every graph of Euler genus $g$ is a subgraph of:
\begin{enumerate}[label=(\alph*)]
\item  $H \boxtimes P \boxtimes K_{\max\{2g,1\}}$ for some graph $H$ of treewidth at most $9$ and for some path $P$;
\item  $H \boxtimes P \boxtimes K_{\max\{2g,3\}}$ for some graph $H$ of treewidth at most $4$ and for some path $P$.
\item $(K_{2g} + H )  \boxtimes P$ for some graph $H$ of treewidth at most $8$  and some path $P$. 
\end{enumerate}
\end{thm}

Here $A+B$ is the complete join of graphs $A$ and $B$. The proof of \cref{SurfaceProduct} uses an elegant cutting lemma to reduce to the planar case. 

\cref{SurfaceProduct} is generalised as follows. A graph $X$ is \emph{apex} if $X-v$ is planar for some vertex $v$. 

\begin{thm}[\citep{DJMMUW19}] 
\label{ApexMinorFree}
For every apex graph $X$, there exists $c\in\mathbb{N}$ such that every $X$-minor-free graph $G$ is a subgraph of $ H\boxtimes P$ for some graph $H$ of treewidth at most $c$  and some path $P$. 
\end{thm}

The proof of \cref{ApexMinorFree} is based on the Graph Minor Structure Theorem of \citet{RS-XVI} and in particular a strengthening of it by \citet{DvoTho}. 

For an arbitrary proper minor-closed class, apex vertices are unavoidable; in this case \citet{DJMMUW19} proved the following product structure theorem. 

\begin{thm}[\citep{DJMMUW19}] 
\label{MinorProduct}
For every proper minor-closed class $\GG$ there exist $k,a\in\mathbb{N}$ such that every graph $G\in \GG$ can be obtained by clique-sums of graphs $G_1,\dots,G_n$ such that for 
 $i\in\{1,\dots,n\}$, 
$$G_i \subseteq (H_i  \boxtimes P_i ) + K_a,$$ 
for some graph $H_i$ with treewidth at most $k$ and some path $P_i$. 
\end{thm}

If we assume bounded maximum degree, then apex vertices in the Graph Minor Structure Theorem can be avoided, which leads to the following theorem of \citet{DEMWW}. 

\begin{thm}[\citep{DEMWW}] 
\label{MinorFreeDegreeStructure}
For every proper minor-closed class $\GG$, every graph in $\GG$ with maximum degree $\Delta$ is a subgraph of $H\boxtimes P$ for some graph $H$ of treewidth $O(\Delta)$ and for some path $P$. 
\end{thm}

It is worth highlighting the similarity of \cref{MinorFreeDegreeStructure} and the following result of \citet{DO95} (refined in \citep{Wood09}).\cref{DegreeTreewidthStructure} says that graphs of bounded treewidth and bounded degree are subgraphs of the product of a tree and a complete graph of bounded size, whereas \cref{MinorFreeDegreeStructure} says that graphs excluding a minor and with bounded degree are subgraphs of the product of a bounded treewidth graph and a path.

\begin{thm}[\citep{DO95,Wood09}] 
\label{DegreeTreewidthStructure}
Every graph with maximum degree $\Delta\geq 1$ and treewidth at most $k\geq 1$ is a subgraph of $T\boxtimes K_{18k\Delta}$ for some tree $T$. 
\end{thm}

\section{Non-Minor Closed Classes}
\label{NonMinorClosedClasses}

A recent direction pursued by \citet{DMW19b} studies graph product structure theorems for various non-minor-closed graph classes. First consider graphs that can be drawn on a surface of bounded genus and with a bounded number of crossings per edge. A graph is \emph{$(g,k)$-planar} if it has a drawing in a surface of Euler genus at most $g$ such that each edge is involved in at most $k$ crossings. Even in the simplest case, there are $(0,1)$-planar graphs that contain arbitrarily large complete graph minors \citep{DEW17}. 

\begin{thm}[\citep{DMW19b}] 
\label{gkPlanarStructure}
Every $(g,k)$-planar graph is a subgraph of $H\boxtimes P$, for some graph $H$ of treewidth $O(gk^6)$ and for some path $P$. 
\end{thm}

Map and string graphs provide further examples of non-minor-closed classes that have product structure theorems. 

Map graphs are defined as follows. Start with a graph $G_0$ embedded in a surface of Euler genus $g$, with each face labelled a `nation' or a `lake', where each vertex of $G_0$ is incident with at most $d$ nations. Let $G$ be the graph whose vertices are the nations of $G_0$, where two vertices are adjacent in $G$ if the corresponding faces in $G_0$ share a vertex. Then $G$ is called a \emph{$(g,d)$-map graph}.  A $(0,d)$-map graph is called a (plane) \emph{$d$-map graph}; see \citep{FLS-SODA12,CGP02} for example. The $(g,3)$-map graphs are precisely the graphs of Euler genus at most $g$; see \citep{DEW17}. So $(g,d)$-map graphs generalise graphs embedded in a surface, and we now assume that $d\geq 4$ for the remainder of this section. 

\begin{thm}[\citep{DMW19b}]
\label{MapPartition}
Every $(g,d)$-map graph is a subgraph of:
\begin{itemize}
\item  $H\boxtimes P\boxtimes K_{O(d^2g)}$, where $H$ is a graph with treewidth at most 14 and $P$ is a path,
\item $H\boxtimes P$, where $H$ is a graph with treewidth $O(gd^2)$ and $P$ is a path.
\end{itemize}
\end{thm}

A \emph{string graph} is the intersection graph of a set of curves in the plane with no three curves meeting at a single point; see  \cite{PachToth-DCG02,FP10,FP14} for example. For $\delta\in\mathbb{N}$, if each curve is in at most $\delta$ intersections with other curves, then the corresponding string graph is called a \emph{$\delta$-string graph}. A \emph{$(g,\delta)$-string} graph is defined analogously for curves on a surface of Euler genus at most $g$.  

\begin{thm}[\citep{DMW19b}] 
\label{StringPartition}
Every $(g,\delta)$-string graph is a subgraph of $H\boxtimes P$, for some graph $H$ of treewidth $O(g\delta^7)$ and some path $P$.
\end{thm}

\cref{gkPlanarStructure,MapPartition,StringPartition} all follow from a more general result of \citet{DMW19b}. A collection $\mathcal{P}$ of paths in a graph $G$ is a \emph{$(k,d)$-shortcut system} (for $G$) if:
\begin{itemize}
\item every path in $\mathcal{P}$ has length at most $k$, and
\item for every $v\in V(G)$, the number of paths in $\mathcal{P}$ that use $v$ as an internal vertex is at most $d$.
\end{itemize} 
Each path $P\in\mathcal{P}$ is called a \emph{shortcut}; if $P$ has endpoints $v$ and $w$ then it is a \emph{$vw$-shortcut}. Given a graph $G$ and a $(k,d)$-shortcut system $\mathcal{P}$ for $G$, let $G^{\mathcal{P}}$ denote the supergraph of $G$ obtained by adding the edge $vw$ for each $vw$-shortcut in $\mathcal{P}$. 

\begin{thm}[\citep{DMW19b}]
\label{ShortcutProduct}
Let $G$ be a subgraph of $H\boxtimes P$, for some graph $H$ of treewidth at most $t$ and for some path $P$. 
Let $\PP$ be a $(k,d)$-shortcut system for $G$. Then $G^\PP$ is a subgraph of $J\boxtimes P'$ for some graph $J$ of treewidth at most $d(k^3+3k)\binom{k+t}{t}-1$ and some path $P'$. 
\end{thm}

\cref{gkPlanarStructure,MapPartition,StringPartition} are then proved by simply constructing a shortcut system. For example, by adding a dummy vertex at each crossing, \citet{DMW19b} noted that every $(g,k)$-planar graph is a subgraph of $G^\PP$ for some graph $G$ of Euler genus at most $g$ and for some $(k+1,2)$-shortcut system $\PP$ for $G$. 

Powers of graphs can also be described by a shortcut system. The \emph{$k$-th power} of a graph $G$ is the graph $G^k$ with vertex set $V(G^k):=V(G)$, where $vw\in E(G^k)$ if and only if $d_G(v,w)\leq k$. \citet{DMW19b} noted that if a graph $G$ has maximum degree $\Delta$, then $G^k = G^\PP$ for some $(k,2k\Delta^{k})$-shortcut system $\PP$. \cref{ShortcutProduct} then implies:

\begin{thm}[\citep{DMW19b}] 
\label{kPowerBasic}
For every graph $G$ of Euler genus $g$ and maximum degree $\Delta$, the $k$-th power $G^k$  is a subgraph of $H\boxtimes P$, for some graph $H$ of treewidth $O(g\Delta^{k} k^{8})$ and some path $P$.
\end{thm}

\section{Polynomial Growth}
\label{PolynomialGrowth}

This section discusses graph classes with polynomial growth. A graph class $\GG$ has \emph{polynomial growth} if for some constant $c$, for every graph $G\in\GG$, for each $r\geq 2$ every $r$-ball in $G$ has at most $r^c$ vertices. 
For example, every $r$-ball in an $n\times n$ grid graph is contained in a $(2r+1)\times(2r+1)$ subgrid, which has size $(2r+1)^2$; therefore the class of grid graphs has polynomial growth. More generally, let $\mathbb{Z}^d$ be the strong product of $d$ infinite two-way paths. That is, $V(\mathbb{Z}^d)=\{(x_1,\dots,x_d): x_1,\dots,x_d\in\mathbb{Z}\}$ where distinct vertices $(x_1,\dots,x_d)$ and $(y_1,\dots,y_d)$ are adjacent in $\mathbb{Z}^d$ if and only if $|x_i-y_i|\leq 1$ for each $i\in\{1,\dots,d\}$. Then every $r$-ball in $\mathbb{Z}^d$ has size at most $(2r+1)^d$. \citet{KL07} characterised the graph classes with polynomial growth as the subgraphs of $\mathbb{Z}^d$. 

\begin{thm}[\citep{KL07}]
Let $G$ be a graph such that for some constant $c$ and for every integer $r\geq 2$, every $r$-ball in $G$ has at most $r^c$ vertices. Then $G\subseteq \mathbb{Z}^{O(c\log c)}$. 
\end{thm}

We show that a seemingly weaker condition also characterises graph classes with polynomial growth. (We emphasise that in \cref{PolyGrowth}, $H_1$ does not necessarily have bounded maximum degree.)\  

\begin{thm}
\label{PolyGrowth}
The following are equivalent for a class of graphs $\GG$: 
 \begin{enumerate}[label=(\arabic*)]
\item $\GG$ has polynomial growth,
\item there exists $d\in\mathbb{N}$ such that every graph in $\GG$ is a subgraph of $\mathbb{Z}^d$, 
\item there exist $d,k,\ell,\Delta\in\mathbb{N}$ such that for every graph $G\in \GG$  there exist graphs $H_1,\dots,H_d$ such that:
 \begin{itemize}
\item $G$ has maximum degree $\Delta$, 
\item $\pw(H_i) \leq k$ for each $i\in\{1,\dots,k\}$, 
\item $H_i$ has maximum degree at most $\ell$ for each $i\in\{2,\dots,d\}$, 
\item $G\subseteq H_1\boxtimes H_2\boxtimes \dots \boxtimes H_d$.
\end{itemize}
\end{enumerate}
\end{thm}

\begin{proof}
\citet{KL07} proved that (1) and (2) are equivalent. It is immediate that (2) implies (3) with $k=1$ and $\ell=2$ and $\Delta=3^d-1$. So it suffices to show that (3) implies (1). Consider graphs $G\in\GG$ and $H_1,\dots,H_d$ satisfying (3). For $i\in\{2,\dots,d\}$, by \cref{SizeBall} below (with $d=0$), every $r$-ball in $H_i$ has at most $(1+\ell)^k (2r+1)^{ k+1 }$ vertices. 
By the result of \citet{KL07}, $H_i\subseteq \mathbb{Z}^{c}$ for some $c=c(k,\ell)$. Thus 
$$G\subseteq H_1 \boxtimes \mathbb{Z}^{c(d-1)}.$$ 
By \cref{SizeBall} again, every $r$-ball in $G$ has size at most 
$$(1+\Delta)^k (2r+1)^{ (k+1) (c(d-1)+1) },$$
which is at most $r^{c'}$ for some $c'=c'(c,\Delta,k)$ and $r\geq 2$. 
Hence (1) holds. 
\end{proof}

\begin{lem}
\label{SizeBall}
For every graph $H$ with pathwidth at most $k\in\mathbb{N}_0$,  for every connected subgraph $G$ of $H \boxtimes \mathbb{Z}^d$ with radius at most $r$ and maximum degree at most $\Delta$, 
$$| V(G) | \leq (1+\Delta)^k (2r+1)^{ (k+1) (d+1) }.$$
\end{lem}

\begin{proof}
The BFS spanning tree of $G$ rooted at the centre of $G$ has radius at most $r$. So it suffices to prove the result when $G$ is a tree. We proceed by induction on $k\geq 0$ with the following hypothesis: For every graph $H$ with pathwidth at most $k\in\mathbb{N}_0$, for every subtree $T$ of $H \boxtimes \mathbb{Z}^d$ with radius at most $r$ and maximum degree at most $\Delta$, 
$$| V(T) | \leq (1+\Delta)^k (2r+1)^{ (k+1) (d+1) }.$$

Since $T$ is connected, we may assume that $H$ is connected. Since $T$ has radius at most $r$,  $$T \subseteq H \boxtimes P_1 \boxtimes \dots \boxtimes P_d,$$
where each $P_i$ is a path on $2r+1$ vertices. 

In the base case $k=0$, we have $H=K_1$ and $T \subseteq P_1 \boxtimes \dots \boxtimes P_d$, implying 
$$|V(T)| \leq (2r+1)^d \leq (1+ \Delta )^0 (2r+1)^{(0+1) (d+1) }.$$ 

Now assume that $k\geq 1$ and the claim holds for $k-1$. 
Let $\widetilde{T}$ be the projection of $V(T) $ into $H$. 
Let $(X_1,\dots,X_{n})$ be a path decomposition of $H$ with width $\pw(H)$. 
We may delete any bag $X_j$ such that $X_j\cap \widetilde{T} = \emptyset$. 
Now assume that $X_1 \cap \widetilde{T} \neq\emptyset$ and $X_n \cap \widetilde{T} \neq\emptyset$.  
Let $x$ be a vertex in $X_1 \cap \widetilde{T}$, and let  $y$ be a vertex in $X_n \cap \widetilde{T}$. 
Thus $(x,x_1,\dots,x_d)\in V(T)$ and $(y,y_1,\dots,y_d)\in V(T)$ for some $x_i,y_i\in V(P_i)$. 
Let $P$ be the path in $T$ with endpoints $(x,x_1,\dots,x_d)$ and $(y,y_1,\dots,y_d)$. 
Since $T$ has radius at most $r$, $P$ has at most $2r+1$ vertices. 
Let $\tilde{P}$ be the set of vertices $v \in V(H)$ such that $(v,z_1,\dots,z_d)\in V(P)$ where $z_i\in V(P_i)$.
Thus $|\tilde{P}| \leq 2r+1$. 
By the choice of $x$ and $y$, we have $\tilde{P}\cap X_j\neq\emptyset$ for each $j\in\{1,\dots,n\}$. 
Let $H':= H-\tilde{P}$. 
Thus $(X_1\setminus \tilde{P}, \dots,X_n\setminus \tilde{P})$ is a path decomposition of $H'$ with width at most $\pw(H)-1$. 
Let $R := \{ (v,z_1,\dots,z_d) : v\in \tilde{P}, z_i\in V(P_i), i\in\{1,\dots,d\} \}$. 
Thus $|R| \leq (2r+1)^{d+1}$. 
Let $T' := T-R$. 
Hence $T'$ is a subgraph of $H' \boxtimes P_1\boxtimes \dots \boxtimes P_d$. 
Each component of $T'$ has a neighbour in $R$, implying that $T'$ has at most $\Delta\,|R|$ components. 
Every subtree of $T$ has radius at most $r$ (centred at the vertex closest to the centre of $T$). 
By induction, each component of $T'$ has at most $(1+\Delta)^{k-1} (2r+1)^{k (d+1) }$ vertices.  
Thus
\begin{align*}
|V(T)| 
& \leq |R| + \Delta  |R| \, (1+\Delta)^{k-1} (2r+1)^{ k(d+1) }\\
& = |R| \, (1 + \Delta (1+\Delta)^{k-1} (2r+1)^{ k(d+1) } ) \\
& \leq |R| \, (1 + \Delta) (1+\Delta)^{k-1} (2r+1)^{ k(d+1) }  \\
& \leq (2r+1)^{d+1}  (1 + \Delta)^{k} (2r+1)^{ k(d+1) } \\
& =  (1 + \Delta)^{k} (2r+1)^{ (k+1)(d+1) }  ,
 \end{align*}
 as desired.
\end{proof}

Property (3) in \cref{PolyGrowth} is best possible in a number of respects. First, note that we cannot allow $H_1$ and $H_2$ to have unbounded maximum degree. For example, if $H_1$ and $H_2$ are both $K_{1,n}$, then $H_1$ and $H_2$ both have pathwidth 1, but $K_{1,n} \boxtimes K_{1,n}$ contains $K_{n,n}$ as a subgraph, which contains a complete binary tree of $\Omega(\log n)$ height, which is a bounded-degree graph with exponential growth. Also, bounded pathwidth cannot be replaced by bounded treewidth, again because of the complete binary tree.

\section{Polynomial Expansion}
\label{PolynomialExpansion}

This section characterises when $\GG_1\boxtimes \GG_2$ has polynomial expansion. Separators are a key tool here.  A \emph{separation} in a graph $G$ is a pair $(G_1,G_2)$ of subgraphs of $G$ such that $G=G_1\cup G_2$ and $E(G_1)\cap E(G_2)=\emptyset$. The \emph{order} of $(G_1,G_2)$ is $|V(G_1) \cap V(G_2)|$. A separation $(G_1,G_2)$ is \emph{balanced} if $|V(G_1)\setminus V(G_2)| \leq \frac{2}{3} |V(G)|$ and $|V(G_2)\setminus V(G_1)| \leq \frac{2}{3} |V(G)|$. 	
A graph class $\GG$ 
 admits \emph{strongly sublinear separators} if there exists $ c \in\mathbb{R}^+$ and $\beta\in[0,1)$ such that for every graph $G\in\GG$, every subgraph $H$ of $G$ has a balanced separation of order at most $c|V(H)|^\beta$. \citet{DN16} noted that a result of \citet{PRS94} implies that graph classes with polynomial expansion admit strongly sublinear separators. \citet{DN16} proved the converse (see \citep{Dvorak16,Dvorak18,ER18} for more results on this theme).

\begin{thm}[\citep{DN16}] 
\label{PolyExpSSS}
A hereditary class of graphs admits strongly sublinear separators if and only if it has  polynomial expansion.
\end{thm}

\citet{RS-II} established the following connection between treewidth and balanced separations. 

\begin{lem}[{\protect\citep[(2.6)]{RS-II}}]
\label{SeparatorLemma}
Every graph $G$ has a balanced separation of order at most $\tw(G)+1$.
\end{lem}

\citet{DN19} proved the following converse. 

\begin{lem}[\citep{DN19}]
\label{ConverseSeparatorLemma}
If every subgraph of a graph $G$ has a balanced separation of order at most $s$, then $\tw(G) \leq 15s$.
\end{lem}

We have the following strongly sublinear bound on the treewidth of graph products. 

\begin{lem}
\label{ProductTreewidth}
Let $G$ be an $n$-vertex subgraph of $\mathbb{Z}^d \boxtimes H$ for some graph $H$. 
Then 
\[\tw(G) \leq 2  ( \tw(H)+1)^{ 1/(d+1) } (dn)^{d/(d+1)}   -  1\enspace.\]
\end{lem}

\begin{proof}
Let $t:= \tw(H)$. 
For $i\in\{1,\dots,d\}$, let $\langle{V^i_0,V^i_1,\dots}\rangle$ be the layering of $G$ determined by the $i$-th dimension. 
Let \[m:= \ceil*{ \left(\frac{dn}{t+1} \right)^{1/(d+1)} } \enspace.\]
For $i\in\{1,\dots,d\}$ and $\alpha\in\{0,\dots,m-1\}$, 
let $V^{i,\alpha} := \bigcup\{ V^i_j : j\equiv \alpha \pmod{m} \}$. 
Thus $V^{i,0},\dots,V^{i,m-1}$ is a partition of $V(G)$. 
Hence $|V^{i,\alpha_i}| \leq \frac{n}{m}$ for some $\alpha_i\in\{0,\dots,m-1\}$. 
Let $X:= \bigcup_{i=1}^d V^{i,\alpha_i}$. Thus $|X| \leq \frac{dn}{m}$. 
Note that each component of $G-X$ is a subgraph of $Q^{d} \boxtimes H$, where $Q$ is the path on $m-1$ vertices. Since $\tw(G)$ equals the maximum treewidth of the connected components of $G$, we have 
$\tw(G) \leq \tw( Q^{d} \boxtimes H ) + |X| $. To obtain a tree decomposition of $Q^{d} \boxtimes H$ with width $(t+1) (m-1)^d - 1$, start with an optimal tree decomposition of $H$, and replace each instance of a vertex of $H$ by the corresponding copy of $Q^{d}$. Thus
\begin{align*}
\tw(G) \leq  (t+1) (m-1)^d - 1 + \frac{dn}{m}
 \leq 2  (t+1)^{ 1/(d+1) } (dn)^{d/(d+1)}   -  1.\qquad\qquad\qedhere
\end{align*}
\end{proof}

\cref{ProductTreewidth} is generalised by our next result, which characterises when a graph product has polynomial expansion. The following definition is key. Say that graph classes $\GG_1$ and $\GG_2$ have \emph{joint polynomial growth} if there exists a polynomial function $p$ such that for every $r\in\mathbb{N}$, there exists $i\in \{1,2\}$ such that for every graph $G\in\GG_i$ every $r$-ball in $G$ has size at most $p(r)$.

\begin{thm}
The following are equivalent for hereditary graph classes $\GG_1$ and $\GG_2$:
 \begin{enumerate}[label=(\arabic*)]
\item $\GG_1 \boxtimes \GG_2$ has polynomial expansion,
\item $\GG_1 \square \GG_2$ has polynomial expansion,
\item $\GG_1$ has polynomial expansion, $\GG_2$ has polynomial expansion, and $\GG_1$ and $\GG_2$ have joint polynomial growth.
\end{enumerate}
\end{thm}

\begin{proof}
(1) implies (2) since $\GG_1\square\ \GG_2 \subseteq \GG_1\boxtimes \GG_2$.

We now show that  (2) implies (3). Assume that $\GG_1 \square \GG_2$ has polynomial expansion. That is, for some polynomial $g$, for every graph $G\in\GG_1\square \GG_2$, every $r$-shallow minor of $G$ has average degree at most $g(r)$. Since $\GG_1\cup \GG_2 \subseteq \GG_1 \square \GG_2$, both $\GG_1$ and $\GG_2$ have polynomial expansion. 

Assume for the sake of contradiction that  $\GG_1$ and $\GG_2$ do not have joint polynomial growth. Thus for every polynomial $p$ there exists $r\in\mathbb{N}$ such that for each $i\in \{1,2\}$ some $r$-ball of some graph $G_i\in\GG_i$ has at least $p(r)$ vertices.
Apply this where $p$ is a polynomial with $p(r) \geq \max\{ 1 + r n, \binom{n}{2} \} $, where  $n:= \ceil{ g(2r)+2}$. Since $\GG_1$ and $\GG_2$ are hereditary, there exists $r\in\mathbb{N}$ such that there is a graph $G_1\in\GG_1$ with radius at most $r$ and at least $1 + rn$ vertices, and there is a graph $G_2\in\GG_2$ with radius at most $r$ and at least $\binom{n}{2}$ vertices. 

Let $z$ be the central vertex in $G_1$. Since $|V(G_1)| \geq 1 + rn$, for some $i\in\{1,\dots,r\}$, there is a set $A$ of $n$ vertices in $G_1$ at distance exactly $i$ from $z$. For all $\{v,w\}\in\binom{A}{2}$, let $P_{v,w}$ be the shortest $vw$-path contained with the union of a shortest $vz$-path and a shortest $wz$-path in $G_1$. Thus $P_{vw}$ has length at most $2r$ and $V(P_{v,w}) \cap A = \{v,w\}$. Let $B$ be a set of $\binom{n}{2}$ vertices in $G_2$. Fix an arbitrary bijection $\sigma :\binom{A}{2} \to B$. 

Let $G:= G_1\square G_2$. For each $v\in A$, let $X_v := G[  \{ (v,x): x\in V(G_2) \} ]$; note that $X_v$ is isomorphic to $G_2$, and thus has radius at most $r$. Moreover, $X_v$ and $X_w$ are disjoint for distinct $v,w\in A$. For  $\{v,w\} \in \binom{A}{2}$, let $Y_{v,w} := G[  \{ (x,\sigma((v,w))): x\in V(G_1) \} ]$; note that $Y_{v,w}$ is isomorphic to $G_1$. Let $Q_{v,w}$ be the copy of the path $P_{v,w}$ within $Y_{v,w}$.
Since $V(P_{v,w}) \cap A = \{v,w\}$ , the only vertices of $Q_{v,w}$ in $\bigcup_{u\in A} X_u$ are $(v,\sigma((v,w)))$ and $(w,\sigma((v,w)))$, which are the endpoints of $Q_{v,w}$ in $X_v$ and $X_w$ respectively. Since $P_{v,w}$ has length at most $2r$, so does $Q_{v,w}$. 

By construction, $Q_{v,w}$ and $Q_{p,q}$ are disjoint for distinct $\{v,w\},\{p,q\} \in \binom{A}{2}$. Contract $X_v$ to a vertex for each $v\in A$, and contract $Q_{v,w}$ to an edge for each $\{v,w\}\in\binom{A}{2}$. We obtain the complete graph $K_n$ as a minor in $G$. Moreover, the minor is $2r$-shallow. This is a contradiction, since $K_n$ has average degree greater than $g(2r)$.

We prove that (3) implies (1) by a series of lemmas below (culminating in \cref{punchline} below). 
\end{proof}

For a graph $G$, a set $X\subseteq V(G)$ is \emph{$r$-localising} if for every component $C$ of $G-X$, there exists a vertex $v\in V(G)$
such that $d_G(u,v)<r$ for every $u\in C$ (note that the distance is in $G$,
not in $G-X$).  

The following is a variation on Lemma~5.2 of \citet{KL07}. For $r\in\mathbb{N}$ and  $p,q\in\mathbb{R}$ with $0<p,q<1$, consider the following function $f_{r,p,q}$ defined on $ \{0,1,\dots,r\}$. First, let $f_{r,p,q}(r):=p$. Now, for every integer $s\in\{0,1,\dots,r-1\}$, inductively define
$$f_{r,p,q}(s):=\min(q\,f_{r,p,q}(\{s+1,\ldots,r\}),1-f_{r,p,q}(\{s+1,\ldots,r\})),$$
where $f(S):=\sum_{i\in S} f(i)$.

\begin{lem}\label{lemma-prob}
Fix $r\in\mathbb{N}$ and $p,q\in\mathbb{R}$ with $0<p,q<1$, such that $f_{r,p,q}(\{0,1,\dots,r\})=1$ (so $f_{r,p,q}$ defines a probability distribution on $\{0,1,\dots,r\}$). For every graph $G$, there exists a probability distribution
over the $r$-localising subsets of $V(G)$ such that  the set $X$ drawn from this distribution satisfies $\mathbb{P}[v\in X]\le p|N^r(v)|+q$ for every $v\in V(G)$.
\end{lem}

\begin{proof}
Let $V(G)=\{v_1,\ldots, v_n\}$. 
For $i\in\{1,\ldots,n\}$, choose $r_i\in \{0,1,\ldots,r\}$ independently at random such that $\mathbb{P}[r_i=s]=f_{r,p,q}(s)$. 
For each $x\in V(G)$, let $i(x)$ be the minimum index $i$ such that $d(x,v_i)\le r_i$, and let $X:=\{x\in V(G): d(x,v_{i(x)})=r_{i(x)}\}$.

First we argue that $X$ is $r$-localising.  Consider any component $C$ of $G-X$,
and let $z$ be the vertex of $C$ with $i(z)$ minimum.  Suppose for the sake of contradiction that $C$ contains a vertex $u$ at distance at least $r$ from $v_{i(z)}$, and let $P$ be a path from $z$ to $u$ in $C$.
Then $P$ contains a vertex $x$ at distance exactly $r_{i(z)}$ from $v_{i(z)}$.  However, since
$i(x)\ge i(z)$, we conclude $i(x)=i(z)$ and $x\in X$, which is a contradiction.

Next, we bound the probability that a vertex $v$ of $G$ is in $X$. Consider any $i\in\{1,\ldots,n\}$.
If $d(v,v_i)>r$, then $\mathbb{P}[i(v)=i]=0$.
If $d(v,v_i)=r$, then $\mathbb{P}[i(v)=i]\le \mathbb{P}[r_{i(v)}=r]=p$.
If $d(v,v_i)<r$, then letting $s:=d(v,v_i)$ we have
\begin{align*}
\mathbb{P}[v\in X|i(v)=i]
& = \mathbb{P} [ r_i=s \,|\, r_1<d(v,v_1), \dots, r_{i-1}<d(v,v_{i-1}), r_i \geq s ] \\
& = \mathbb{P}[r_i=s \,|\, r_i\ge s]\\
& = \frac{f_{r,p,q}(s)}{f_{r,p,q}(\{s,\ldots,r\})}\\
& \le q.
\end{align*}
Therefore,
\begin{equation*}
\mathbb{P}[v\in X]=\sum_{i=1}^n \mathbb{P}[i(v)=i]\cdot\mathbb{P}[v\in X|i(v)=i]\le p|N^r(v)|+q.\qedhere
\end{equation*}
\end{proof}

\begin{cor}\label{cor-frac}
For every polynomial $g$, there exists $r_0$ such that the following holds.
Let $r\ge r_0$ be a positive integer and let $G$ be a graph such that $|N^r(v)|\le g(r)$ for every $v\in V(G)$.  Then there exists a probability distribution
over the $r$-localising subsets of $V(G)$ such that the set $X$ drawn from this distribution satisfies $\mathbb{P}[v\in X]\le 2r^{-1/2}$ for every $v\in V(G)$.
\end{cor}

\begin{proof}
Let $c$ be the degree of $g$ plus one, so that $g(r)\le r^c$ for every sufficiently large $r$. Let $p:=r^{-c-1/2}$ and $q:=r^{-1/2}$.  Note that for sufficiently large $r$, 
$$p(1+q)^r\ge pe^{qr/2}=\exp(\sqrt{r}/2-(c+1/2)\log r)>1.$$
Hence $f_{r,p,q}(r) = p > 1 / (q+1)^r$. 
It follows by induction that $f_{r,p,q}(\{s,\dots,r\}) \geq 1 / (q+1)^s$ for each $s\in\{1,\dots,r\}$. 
Thus $f_{r,p,q}(0)=1-f_{r,p,q}(\{1,\dots,r\})$ and  $f_{r,p,q}(\{0,\dots,r\})=1$. 
The claim follows from \cref{lemma-prob}.
\end{proof}

\begin{cor}\label{cor-frac-weight}
For every polynomial $g$, there exists $r_0$ such that the following holds. Let $r\ge r_0$ be a positive integer and let $G$ be a graph such that $|N^r[v]|\le g(r)$ for every $v\in V(G)$. Then for every function $w:V(G)\to\mathbb{R}_0^+$, there exists $X\subseteq V(G)$ such that $w(X)\le 2r^{-1/2}w(V(G))$ and
each component of $G-X$ has at most $g(r)$ vertices.
\end{cor}

\begin{proof}
Choose an $r$-localising set $X\subseteq V(G)$ using \cref{cor-frac}. Since $X$ is $r$-localising and $|N^r[v]|\le g(r)$ for every $v\in V(G)$, each component of $G-X$ has at most $g(r)$ vertices.  Furthermore, 
\begin{equation*}
\mathbb{E}[|w(X)|]=\sum_{v\in V(G)} \mathbb{P}[v\in X]w(v)\le 2r^{-1/2}w(V(G)).
\end{equation*}
Hence there is a choice for $X$ such that $w(X)\le 2r^{-1/2}w(V(G))$. 
\end{proof}

\begin{lem}
\label{punchline}
Suppose $\GG_1$ and $\GG_2$ are classes with strongly sublinear separators
and of joint polynomial growth (bounded by a polynomial $g$).  Then $\GG_1\boxtimes \GG_2$ has strongly sublinear separators.
\end{lem}
\begin{proof}
Let $\varepsilon>0$ be such that every subgraph $F$ of a graph from $\GG_1\cup \GG_2$
has a balanced separator of order at most $\lceil |V(F)|^{1-\varepsilon}\rceil$.  Let $\beta>0$ be sufficiently small
(depending on $\varepsilon$ and $g$).

Suppose $G_1\in \GG_1$, $G_2\in\GG_2$, and $H$ is a subgraph of $G_1\boxtimes G_2$. Let $\pi_1$ and $\pi_2$ be the projections from $H$ to $G_1$ and $G_2$.  Let $n :=|V(H)|$ and $r :=n^\beta$. By symmetry, we may assume $|N^r[v]| \le g(r)$ for every vertex $v$ of $G_1$.  Let $w(v) :=|\pi_1^{-1}(v)|$ for each $v\in V(G_1)$. By \cref{cor-frac-weight}, there exists $X\subseteq V(G_1)$ such that $w(X)=O(r^{-1/2}n)=O(n^{1-\beta/2})$
and each component of $G_1-X$ has at most $g(r)=g(n^\beta)=O(n^{\varepsilon/2})$ vertices.
Let $A:=\pi_1^{-1}(X)$; then  $|A|=w(X)=O(n^{1-\beta/2})$.

The graph $G_2$ has treewidth $O(n^{1-\varepsilon})$, and thus the product of $G_2$ with $G_1-X$ (as well as its subgraph $H-A$)
has treewidth $O(n^{1-\varepsilon}g(r))=O(n^{1-\varepsilon/2})$.  Consequently, $H-A$ has a balanced separator $B$ of
order $O(n^{1-\varepsilon/2})$, and $A\cup B$ is a balanced separator of $H$ of order $O(n^{1-\min(\varepsilon,\beta)/2})$.
\end{proof}

\section{Bounded Expansion}
\label{BoundedExpansion}

This section characterises when $\GG_1\boxtimes \GG_2$ has bounded expansion. The following definition by \citet{KY03} is the key tool. For a graph $G$, linear ordering $\preceq$ of $V(G)$, vertex $v\in V(G)$,
and integer $r\geq1$, a vertex $x$ is \emph{$(r,\preceq)$-reachable} from $v$ if 
there is a path $v=v_0,v_1,\dots,v_{r'}=x$ of length
$r'\in\{0,1,\dots,r\}$ such that $x\preceq v \prec v_i$ for all $i\in\{1,2,\dots,r'-1\}$. 
For a graph $G$ and $r\in\mathbb{N}$, the \emph{$r$-colouring number}
$\col_r(G)$ of~$G$, also known as the \emph{strong $r$-colouring number}, is the minimum integer $k$ such that there is a linear
ordering~$\preceq$ of $V(G)$ such that at most $k$ vertices are $(r,\preceq)$-reachable from each
vertex $v$ of $G$. For example, \citet{HOQRS17} proved that every planar graph $G$ satisfies $\col_r(G) \leq 5r+1$, and more generally, that every $K_t$-minor-free graph $G$ satisfies $\col_r(G) \leq \binom{t-1}{2}(2r+1)$. Most generally, \citet{Zhu09} showed that these $r$-colouring numbers  characterise bounded expansion classes.

\begin{thm}[\citep{Zhu09}]
\label{BE-SCOL}
A graph class $\GG$ has bounded expansion if and only if for each $r\in\mathbb{N}$ there exists $c\in\mathbb{N}$ such that $\col_r (G) \leq c$ for all $G\in\GG$. 
\end{thm}

We now show that if $G$ has bounded $r$-colouring number and $H$ has bounded maximum degree, then $G\boxtimes H$ has bounded $r$-colouring number. 

\begin{lem}
\label{ColProduct}
If $G$ is a graph with  $\col_r(G) \leq c$ and $H$ is a graph with maximum degree at most $\Delta$, then $\col_r(G\boxtimes H) < c (\Delta+2)^r$.
\end{lem}

\begin{proof}
Let $G^+$ and $H^+$ be the pseudographs obtained from $G$ and $H$ by adding a loop at every vertex. 
Let $\preceq_G$ be a vertex-ordering of $G$ witnessing that $\col_r(G) \leq c$. 
Let $\preceq$ be an ordering of $V(G\boxtimes H)$ where $v_1 \prec_G v_2$ implies $(v_1,w_1)\prec (v_2,w_2)$ for all $v_1,v_2\in V(G)$ and $w_1,w_2\in V(H)$. 
We now bound the number of vertices of $G\boxtimes H$ that are $(r,\preceq)$-reachable from a fixed vertex $(v,w)\in V(G\boxtimes H)$. 
Say $(x,y)$ is $(r,\preceq)$-reachable from $(v,w)$. 
Thus there is a path $(v,w)=(v_0,w_0),(v_1,w_1),\dots,(v_{r'},w_{r'})=(x,y)$ of length $r'\in\{0,\dots,r\}$, 
such that $(v_{r'},w_{r'})\preceq (v,w)\prec (v_i,w_i)$ for each $i\in\{1,\dots,r'-1\}$. 
Charge $(x,y)$ to the pair $( x, (w_0,w_1,\dots,w_{r'} ))$. 
By the definition of $\boxtimes$, the sequence $(v_0,v_1,\dots,v_{r'})$ is a walk in $G^+$, and 
the sequence $(w_0,w_1,\dots,w_{r'})$ is a walk in $H^+$. 
By the definition of $\preceq$, we have $v_{r'} \preceq v_0 \preceq v_i$ for each $i\in\{1,\dots,r'-1\}$. 
Thus $v_{r'}$ is $(r,\preceq_G)$-reachable from $v_0$  in $G$. 
By assumption, at most $c$ vertices are $(r,\preceq_G)$-reachable from $v_0$ in $G$. 
The number of walks of length at most $r$ in $H^+$ starting at $w_0$ is at most $\sum_{i=0}^r(\Delta+1)^i < (\Delta+2)^r$.
Thus for each vertex $x \in V(G)$, less than  $(\Delta+2)^r$ vertices $(r,\preceq)$-reachable from $(v,w)$ are charged to some pair $(x,W)$. 
Hence, less than $c(\Delta+2)^r$ vertices in $G\boxtimes H$ are $(r,\preceq_G)$-reachable from $(v,w)$ in $\preceq$. Therefore $\col_r(G\boxtimes H) < c(\Delta+2)^r$.
\end{proof}

The next theorem is the main contribution of this section. 

\begin{thm}
The following are equivalent for hereditary graph classes $\GG_1$ and $\GG_2$:
\begin{enumerate}
\item $\GG_1\boxtimes \GG_2$ has bounded expansion,
\item $\GG_1\square\ \GG_2$ has bounded expansion,
\item 
 both  $\GG_1$ and $\GG_2$ have bounded expansion, and at least one of  $\GG_1$ and $\GG_2$ has bounded maximum degree. 
\end{enumerate}
\end{thm}

\begin{proof}
(1) implies (2) since $G_1\square G_2 \subseteq G_1 \boxtimes G_2$. 

We now show that (2) implies (3). Assume that $\GG_1\square\ \GG_2$ has bounded expansion. 
Since $\GG_1\cup\GG_2 \subseteq \GG_1\boxtimes \GG_2$, both  $\GG_1$ and $\GG_2$ also have bounded expansion. If neither $\GG_1$ nor $\GG_2$ have bounded maximum degree, then for every  $n\in\mathbb{N}$, the star graph $K_{1,n}$ is a subgraph of some graph $G_1\in\GG_1$ and of some graph $G_2\in\GG_2$. Observe that $K_{1,n}\square K_{1,n}$ contains a 1-subdivision of $K_{n,n}$. Thus $G_1\square G_2$ contains a graph with average degree $n$ (namely, $K_{n,n}$) as a 1-shallow minor, which is a contradiction since $\GG_1\square\ \GG_2$ has bounded expansion. Hence  at least one of  $\GG_1$ and $\GG_2$ has bounded maximum degree. 

We now show that (3) implies (1). Assume that  both  $\GG_1$ and $\GG_2$ have bounded expansion, and every graph in $\GG_2$ has maximum degree at most $\Delta$. By \cref{BE-SCOL}, for each $r\in\mathbb{N}$ there exists $c_r\in\mathbb{N}$ such that $\col_r (G) \leq c_r$ for all $G_1\in\GG_1$. Let $G_2\in \GG_2$. By \cref{ColProduct}, we have $\col_r(G_1\boxtimes G_2) \leq c_r (\Delta+2)^r$, and the same bound holds for every subgraph of $G_1\boxtimes G_2$. By \cref{BE-SCOL}, $\GG_1\boxtimes \GG_2$ has bounded expansion.
\end{proof}

\section{Geometric Graph Classes }
\label{HigherDimensions}

The section explores graph product structure theorems for various geometrically defined graph classes.  

The \emph{unit disc} graph of a finite set $X\subseteq \mathbb{R}^d$ is the graph $G$ with $V(G)=X$ where $vw\in E(G)$ if and only if $\dist(v,w)\leq 1$. Here $\dist$ is the Euclidean distance in $\mathbb{R}^d$. Let $\mathbb{Z}^d$ be the strong product of $d$ paths $P\boxtimes \dots \boxtimes P$ (the $d$-dimensional grid graph with crosses). The next result describes unit discs in terms of strong products, which implies that the class of unit disc graphs with bounded dimension and bounded maximum clique size has polynomial growth. 

\begin{thm}
\label{UnitDiscGraph}
Every unit disc graph $G$ in $\mathbb{R}^d$ with no $(k+1)$-clique is a subgraph of 
$\mathbb{Z}^d \boxtimes K_{k \ceil{\sqrt{d}}^d}$.
\end{thm}

\begin{proof}
Let $t:= k \ceil{\sqrt{d}}^d$. 
Let $x_i(v)$ be the $i$-th coordinate of each vertex $v\in V(G)$. 
For $p_1,\dots,p_d\in\mathbb{Z}$, let $V\langle p_1,\dots,p_d\rangle$ be the set of vertices $v\in V(G)$ such that $p_i  \leq x_i(v) < p_i+1$ for each $i\in\{1,\dots,d\}$. Thus the sets  $V\langle p_1,\dots,p_d\rangle$ partition $V(G)$. 
Each set $V\langle p_1,\dots,p_d\rangle$ consists of the set of vertices in a particular unit cube. Note that the unit cube can be partitioned into $\ceil{\sqrt{d}}^d$  sub-cubes, each with side length at most $\frac{1}{\sqrt{d}}$ and thus with diameter at most 1. The set of vertices in a sub-cube with diameter at most 1 is a clique in $G$. Thus at most $k$ vertices lie in a single sub-cube, and  $|V\langle p_1,\dots,p_d\rangle | \leq t$.  Injectively label the vertices in $V\langle p_1,\dots,p_d\rangle$ by $1,2,\dots,t$. 
Map each vertex $v$ in 
$V\langle p_1,\dots,p_d\rangle$ labelled $\ell(v)$ to the vertex  $(p_1,\dots,p_d,\ell(v))$ of 
$\mathbb{Z}^d \boxtimes K_{t}$. 
Thus the vertices of $G$ are mapped to distinct vertices of $\mathbb{Z}^d \boxtimes K_{t}$. 
For each edge $vw\in E(G)$, if $v\in V\langle p_1,\dots,p_d\rangle$ and 
$w\in V\langle q_1,\dots,q_d\rangle$, then $|p_i-q_i|\leq 1$ for each $i\in\{1,\dots,d\}$, and if $p_i=q_i$ for each $i\in\{1,\dots,d\}$, then $\ell(v)\neq\ell(w)$. Thus $v$ and $w$ are mapped to adjacent vertices in 
$\mathbb{Z}^d \boxtimes K_{t}$. 
\end{proof}

By a volume argument, every covering of the unit cube by balls of diameter 1 uses at least $(\frac{d}{18})^{d/2}$ balls. So the $k \ceil{\sqrt{d}}^d$ term in the above theorem cannot be drastically improved.

The \emph{$k$-nearest neighbour graph} of  a finite set $P\subset \mathbb{R}^d$ has vertex set $P$, where two vertices $v$ and $w$ are adjacent if $w$ is the one of the $k$ points in $P$ closest to $v$, or  $v$ is the one of the $k$ points in $P$ closest to $w$.  \citet{MTTV97} showed that such graphs admit separators of order $O(n^{1-1/d})$. 

Can we describe the structure of $k$-nearest-neighbour graphs using graph products? 

\begin{conj} 
Every $k$-nearest neighbour graph in $\mathbb{R}^d$ is a subgraph of $H \boxtimes \mathbb{Z}^{d-1}$ for some graph $H$ with treewidth at most $f(k,d)$.
\end{conj}

This conjecture is trivial for $d=1$ and true for $d=2$, as proved by \citet{DMW19b}. Note that ``treewidth at most $f(k,d)$'' cannot be replaced by ``pathwidth at most $f(k,d)$'' for $d\geq 2$ because complete binary trees are 2-dimensional 2-nearest neighbour graphs without polynomial growth (see \cref{PolyGrowth}).

Here is a still more general example: \citet{MTTV97} defined a \emph{$(d,k)$-neighbourhood system} to consist of a collection $\mathcal{C}$ of $n$ balls in $\mathbb{R}^d$  such that no point in $\mathbb{R}^d$ is covered by more than $k$ balls. Consider the associated graph with one vertex for each ball, where two vertices are adjacent if the corresponding balls intersect.  \citet{MTTV97} showed that such graphs admit balanced separators of order $O(n^{1-1/d})$. Note that by the Koebe circle packing theorem, every planar graph is associated with some $(2,2)$-neighbourhood system. Thus the result of 
\citet{MTTV97} is a far-reaching generalisation of the Lipton-Tarjan Separator Theorem \citep{LT79}. Is there a product structure theorem for these graphs? Might the structure in \cref{MinorProduct} be applicable here?

\begin{open}
If $G$ is the graph associated with a  $(d,k)$-neighbourhood system, can $G$ be obtained from clique-sums of graphs $G_1,\dots,G_n$ such that $G_i \subseteq (H_i  \boxtimes P^{(d-1)} ) + K_a$, for some graph $H_i$ with treewidth at most $k$,  where $a$ is a constant that depends only on $k$ and $d$. The natural choice is $a=k-1$. 
\end{open} 

One can ask a similar question for graphs embeddable in a finite-dimensional Euclidean space with bounded
distortion of distances. \citet{Dvorak18} showed that such graphs have strongly sublinear separators.

\section{Open Problems}
\label{OpenProblems}

We finish the paper with a number of open problems. 

It is open whether the treewidth 4 bound in \cref{SurfaceProduct}(b) can be improved. 

\begin{open}
For every $g\in\mathbb{N}$, does there exist $t\in\mathbb{N}$ such that every graph of Euler genus $g$ is a subgraph of $H\boxtimes P\boxtimes K_t$ for some graph $H$ of treewidth at most $3$?
\end{open}

The proofs of \cref{ApexMinorFree,MinorProduct} both use the Graph Minor Structure Theorem of \citet{RS-XVI}. 

\begin{open}
Is there a proof of \cref{ApexMinorFree} or \cref{MinorProduct} that does not use the graph minor structure theorem?
\end{open}

The following problem asks to minimise the treewidth in \cref{gkPlanarStructure}. 

\begin{open}[\citep{DMW19b}] 
Does there exist a constant $c$ such that for every $k\in\mathbb{N}$ there exists $t\in\mathbb{N}$ such that every $k$-planar graph is a subgraph of $H\boxtimes P\boxtimes K_t$ for some graph $H$ of treewidth at most $c$?
\end{open}

\begin{open}
\label{CharacterisePolynomialExpansion}
Can any graph class with linear or polynomial expansion be described as a product of simpler graph classes along with apex vertices, clique sums, and other ingredients. 
\end{open}

Such a theorem would be useful for proving properties about such classes. Recent results say that the ``other ingredients'' in \cref{CharacterisePolynomialExpansion} are needed, as we now explain. Let
$G'$ be the  $6\tw(G)$-subdivision of a  graph $G$. Let $\mathcal{G}:=\{G':G \text{ is a graph}\}$. \citet{GKRSS18} proved that $\mathcal{G}$ has linear expansion.  On the other hand, \citet{JPP20} proved that there are graphs $G$ such that every $p$-centred colouring of $G'$ has at least $2^{cp^{1/2}}$ colours, for some constant $c>0$. We do not define ``$p$-centred colouring'' here since we do not need the definition. All we need to know is that subgraphs of $H\boxtimes P$, where $H$ has bounded treewidth and $P$ is a path, have $p$-centred colourings with $f(p)$ colours, where $f$ is a polynomial function (see \citep{DMW19b,DFMS20}). This result is easily extended to allow for apex vertices and clique sums. This shows that there are graphs with linear expansion that cannot be described solely in terms of products of bounded treewidth graphs and paths (plus apex vertices and clique sums). For related results, see \citet{DMN}.

One way to test the quality of such a structure theorem is whether they resolve the following questions about queue-number and nonrepetitive chromatic number mentioned in \cref{Intro}:

\begin{open}
Do graph classes with linear or polynomial expansion have bounded queue-number?
\end{open}

\begin{open}
Do graphs classes with linear or polynomial (or even single
exponential) expansion have bounded nonrepetitive chromatic number?
\end{open}

Note that bounded degree graphs are an example with exponential expansion and unbounded queue-number~\citep{Wood-QueueDegree}. Similarly, subdivisions of complete graphs $K_n$ with $o(\log n)$ division vertices per edge are an example with super-exponential expansion and unbounded nonrepetitive chromatic number~\citep{NOW11}. Thus the graph classes mentioned in the above open problems are the largest possible with bounded queue-number or bounded nonrepetitive chromatic number.

\subsection{Algorithmic Questions}
\label{AlgorithmicQuestions}

Do product structure theorems have algorithmic applications? Consider the method of \citet{Baker94} for designing polynomial-time approximation schemes for problems on planar graphs. This method partitions the graph into BFS layers, such that the problem can be solved optimally on each layer (since the induced subgraph has bounded treewidth), and then combines the solutions from each layer. \cref{PlanarProduct} gives a more precise description of the layered structure of planar graphs. It is conceivable that this extra structural information is useful when designing algorithms for planar graphs (and any graph class that has a product structure theorem). 

Some NP-complete problems can be solved efficiently on planar graphs. Can these results be extended to any subgraph of the strong product of a bounded treewidth graph and a path? For example, can max-cut be solved efficiently for graphs that are subgraphs of $H\boxtimes P$, where $H$ is a bounded treewidth graph and $P$ is a
path, such as apex-minor-free graphs?  This would be a considerable
generalisation of the known polynomial-time algorithm for max-cut on
planar graphs \cite{Hadlock75} and on graphs of bounded genus \citep{GL99}. 

Some problems can be solved by particularly fast algorithms on planar graphs. Can such results be generalised for any subgraph of the strong product of a bounded treewidth graph and a path? For example, can shortest paths be computed in $O(n)$ time for $n$-vertex subgraphs of $H\boxtimes P$, where $H$ is a bounded treewidth graph and $P$ is a path?  Can maximum flows be computed in $n \log^{O(1)}(n)$ time for $n$-vertex subgraphs of $H\boxtimes P$? See \citep{Erickson10,CCE13} for analogous results for planar graphs.

  \let\oldthebibliography=\thebibliography
  \let\endoldthebibliography=\endthebibliography
  \renewenvironment{thebibliography}[1]{%
    \begin{oldthebibliography}{#1}%
      \setlength{\parskip}{0.4ex}%
      \setlength{\itemsep}{0.4ex}%
  }{\end{oldthebibliography}}

\bibliographystyle{myNatbibStyle}
\bibliography{myBibliography}

\end{document}